\documentclass[a4paper]{amsart}
\usepackage{amssymb, enumitem}
\usepackage{hyperref, aliascnt}

\newcommand{\andSep}{\,\,\,\text{ and }\,\,\,}
\newcommand{\CC}{{\mathbb{C}}}
\newcommand{\ca}{$C^*$-algebra}

\newenvironment{psmallmatrix}
  {\left(\begin{smallmatrix}}
  {\end{smallmatrix}\right)}

\def\today{\number\day\space\ifcase\month\or   January\or February\or
   March\or April\or May\or June\or   July\or August\or September\or
   October\or November\or December\fi\   \number\year}

\newtheorem{lma}{Lemma}[section]

\newaliascnt{thmCt}{lma}
\newtheorem{thm}[thmCt]{Theorem}
\aliascntresetthe{thmCt}

\newaliascnt{corCt}{lma}
\newtheorem{cor}[corCt]{Corollary}
\aliascntresetthe{corCt}

\newaliascnt{prpCt}{lma}
\newtheorem{prp}[prpCt]{Proposition}
\aliascntresetthe{prpCt}

\theoremstyle{definition}

\newaliascnt{dfnCt}{lma}

\aliascntresetthe{dfnCt}

\newaliascnt{rmkCt}{lma}
\newtheorem{rmk}[rmkCt]{Remark}
\aliascntresetthe{rmkCt}

\newaliascnt{exaCt}{lma}

\aliascntresetthe{exaCt}

\newaliascnt{qstCt}{lma}
\newtheorem{qst}[qstCt]{Question}
\aliascntresetthe{qstCt}

\newaliascnt{pgrCt}{lma}
\newtheorem{pgr}[pgrCt]{}
\aliascntresetthe{pgrCt}

\newcounter{theoremintro}

\newaliascnt{thmIntroCt}{theoremintro}
\newtheorem{thmIntro}[thmIntroCt]{Theorem}
\aliascntresetthe{thmIntroCt}

\newaliascnt{corIntroCt}{theoremintro}

\aliascntresetthe{corIntroCt}

\newaliascnt{qstIntroCt}{theoremintro}

\aliascntresetthe{qstIntroCt}

\title{Prime ideals in C*-algebras and applications to Lie theory}

\author[Eusebio Gardella]{Eusebio Gardella}
\address{Eusebio Gardella
Department of Mathematical Sciences, Chalmers University of
Technology and University of Gothenburg, Gothenburg SE-412 96, Sweden.}
\email{gardella@chalmers.se}
\urladdr{www.math.chalmers.se/~gardella}

\author{Hannes Thiel}
\address{Hannes~Thiel, 
Department of Mathematical Sciences, Chalmers University of Technology and University of
Gothenburg, Gothenburg SE-412 96, Sweden.}
\email{hannes.thiel@chalmers.se}
\urladdr{www.hannesthiel.org}

\thanks{
The first named author was partially supported by the Swedish Research Council Grant 2021-04561.
The second named author was partially supported by the Knut and Alice Wallenberg Foundation (KAW 2021.0140).
}
 	
\subjclass[2010]%
{Primary
16N60, 
46L05. 
Secondary
16W10, 
47B47. 
}
\date{\today}

\begin{document}

\begin{abstract}
We show that every proper, dense ideal in a \ca{} is contained in a prime ideal.
It follows that a subset generates a \ca{} as a not necessarily closed ideal if and only if it is not contained in any prime ideal.

This allows us to transfer Lie theory results from prime rings to \ca{s}.
For example, if a \ca{} $A$ is generated by its commutator subspace $[A,A]$ as a ring, then $[[A,A],[A,A]] = [A,A]$.
Further, given Lie ideals $K$ and $L$ in $A$, then $[K,L]$ generates $A$ as a not necessarily closed ideal if and only if $[K,K]$ and $[L,L]$ do, and moreover this implies that $[K,L]=[A,A]$.

We also discover new properties of the subspace generated by square-zero elements and relate it to the commutator subspace of a \ca.
\end{abstract}

\maketitle

\section{Introduction}

An ideal $I$ in a ring $R$ is said to be \emph{prime} if $I \neq R$ and if whenever $J,K \subseteq R$ are ideals satisfying $JK \subseteq I$ then $J \subseteq I$ or $K \subseteq I$.
Further, an ideal is \emph{semiprime} if $I=R$ or if $I$ is the intersection of prime ideals.
A fundamental result in \ca{s} is that every \emph{closed} ideal is semiprime.
Indeed, every proper closed ideal is the intersection of closed prime ideals.
In particular, a proper closed ideal is contained in a closed prime ideal, and the main result of this paper is a strengthening of this result to ideals that are not necessarily closed.
The following is \autoref{prp:SemiprimeClosure}:

\begin{thmIntro}
\label{ThmA}
Every proper, not necessarily closed, ideal in a \ca{} is contained in a prime ideal.
\end{thmIntro}

If an ideal is not dense, then its closure is proper and therefore even contained in a closed prime ideal.
Thus, \autoref{ThmA} is most interesting for proper, \emph{dense} ideals.

We can view \autoref{ThmA} as a generalization of the basic result that every proper ideal in a \emph{unital} \ca{} is contained in a maximal ideal. 
(Maximal ideals in unital \ca{s} are prime.)
Recently, Lee \cite{Lee22HigherCommutators, Lee22AddSubgpGenNCPoly} initiated the study of Lie ideals in rings where every proper ideal is contained in a maximal ideal. 
We expect that similar results can be shown for rings where every proper ideal is contained in a prime ideal -- and we demonstrate this in \autoref{sec:Lie} for the case of \ca{s}.

\medskip

Herstein's characterization of Lie ideals in simple rings \cite{Her55LieJordanSimpleRing} was the starting point for the development of a beautiful theory of Lie ideals in prime and semiprime rings
\cite{Bax65CommutatorSubgroupRing, Her69TopicsRngThy, Her70LieStructure, LanMon72LieStrPrimeChar2}.
Using that \ca{s} are semiprime rings, some of the Lie theory applies directly in this setting.
For example, if $A$ is a unital, simple \ca{}, then a subspace $L \subseteq A$ is a Lie ideal (see \autoref{pgr:LieThy} for definitions) if and only if $L=\{0\}$, $L=\CC 1$ or $[A,A] \subseteq L$.
By Pop's theorem \cite{Pop02FiniteSumsCommutators}, $A=[A,A]$ if and only if $A$ admits no tracial states, and so in a unital, simple \ca{} $A$ without tracial states, the only Lie ideals are $\{0\}$, $\CC 1$ and $A$.
Further, if $A$ has a unique tracial state $\tau$, then $\ker(\tau) = \overline{[A,A]}$ and $A/\overline{[A,A]} \cong \CC$ by work of Cuntz-Pedersen \cite{CunPed79EquivTraces}, and thus the only \emph{closed} Lie ideals in $A$ are $\{0\}$, $\CC 1$, $\overline{[A,A]}$ and $A$;
see \cite[Theorem~2.5]{MarMur98UniInvSpaceCAlg}.

More generally, if $L$ is a Lie ideal in a \ca{} $A$  then either $L$ is contained in the center $Z(A)$, or there exists a nonzero (not necessarily closed) ideal $I \subseteq A$ such that $[A,I] \subseteq L$;
see \cite{Her70LieStructure}.
If $I$ is semiprime, then one can pass to the quotient~$A/I$ and continue the analysis there.

Using that \emph{closed} ideals in \ca{s} are semiprime, a comprehensive theory for \emph{closed} Lie ideals in \ca{s} was developed by Miers \cite{Mie81ClosedLie}, Bre\v{s}ar-Kissin-Shulman \cite{BreKisShu08LieIdeals} and Robert \cite{Rob16LieIdeals}.
Based on \autoref{ThmA}, we are able to transfer some results from the Lie theory of prime rings to the study of not necessarily closed Lie ideals in \ca{s}.
The following is \autoref{prp:Lie}:

\begin{thmIntro}
\label{ThmB}
Given a Lie ideal $L$ in a \ca{} $A$, the Lie ideal $[A,L]$ generates~$A$ as a not necessarily closed ideal if and only if $[L,L]$, or equivalently $[[L,L],[L,L]]$, does.
Moreover, if this is the case, then $[A,A] = [L,L] \subseteq L$.
\end{thmIntro}

As an interesting special case of \autoref{ThmB} we can consider $A$ as a Lie ideal in itself, and we deduce that $[A,A]=[[A,A],[A,A]]$ whenever $A$ is generated by~$[A,A]$ as a not necessarily closed ideal;
see \autoref{prp:HigherCommutator}.

Generalizing \autoref{ThmB} to the case of two Lie ideals, we have \autoref{prp:LieTwo}:

\begin{thmIntro}
Given Lie ideals $K$ and $L$ in a \ca{} $A$, the Lie ideal $[K,L]$ generates $A$ as a not necessarily closed ideal if and only if $[K,K]$ and $[L,L]$ do.
Moreover, if this is the case, then $[A,A] = [K,L]$.
\end{thmIntro}

In \autoref{sec:N2}, we study the connection between the commutator subspace $[A,A]$ and the subspace $N$ generated by the set of square-zero elements in a \ca{}~$A$.
Robert proved that $[A,A] = N$ if $A$ is unital and has no characters;
see \cite[Theorem~4.2]{Rob16LieIdeals}.
We generalized this to the case that $A$ is a zero-product balanced \ca{} (\cite[Theorem~5.3]{GarThi23ZeroProdBalanced}), which includes all \ca{s} whose multiplier algebra has no characters (\cite{GarThi23pre:ZeroProdRingsCAlgs}).

In \cite[Question~2.5]{Rob16LieIdeals} Robert asks if $[A,A] = N$ holds for every \ca.
In general, it is known that every square-zero element is a commutator, and thus $N \subseteq [A,A]$;
see, for example, \cite[Lemma~2.1]{Rob16LieIdeals}.
We substantially sharpen this result.
The following is \autoref{prp:N}:

\begin{thmIntro}
\label{ThmD}
Let $N$ denote the not necessarily closed subspace generated by the square-zero elements in a \ca.
We have $[N,N] = N$.
\end{thmIntro}

It is known that $[A,A]$ and $N$ have the same norm-closure;
see \cite[Proposition~2.2]{AlaExtVilBreSpe16CommutatorsSquareZero}, \cite[Corollary~2.3]{Rob16LieIdeals}.
Using this in combination with \autoref{ThmD}, we also recover the result that $[[A,A],[A,A]]$ and $[A,A]$ have the same closure;
see \autoref{prp:ClosureHigherCommutators}.

\medskip

In forthcoming work \cite{GarKitThi23pre:SemiprimeIdls} with Kitamura we give a characterization of semiprime ideals in \ca{s}, and use it to deduce that every semiprime ideal is self-adjoint. 
In particular, it follows that all prime ideals in \ca{s} are \emph{automatically} self-adjoint. 
(Although not all ideals in \ca{s} are self-adjoint.)

It would also be interesting to explore to what extent the results of this paper can be generalized to the setting of $L^p$-operator algebras;
see \cite{Gar21ModernLp} for an introduction to the subject

With view towards \autoref{rmk:NonSelfadjoint}, one has to restrict to $L^p$-operator algebras that `look like \ca{s}', for example group $L^p$-operator algebras of nondiscrete groups \cite{GarThi15GpAlgLp, GarThi19ReprConvLq, GarThi22IsoConv} or groupoid $L^p$-operator algebras \cite{GarLup17ReprGrpdLp, ChoGarThi19arX:LpRigidity}.

\subsection*{Acknowledgements}

The authors thank Leonel Robert for valuable comments on earlier versions of this paper.

\section{Prime ideals in \texorpdfstring{$C^*$}{C*}-algebras}
\label{sec:PrimeIdeals}

Throughout this paper, by an `ideal' in a \ca{} we mean a not necessarily closed, two-sided ideal.
To avoid confusion, we will nevertheless usually clarify whether a considered ideal in a \ca{} is assumed to be closed or not.

In this section, we prove that \ca{s} are never radical extensions over proper (not necessarily closed) ideals;
see \autoref{prp:NotRadicalExtension}.
It follows that if $I$ is a proper ideal of a \ca{}, then the semiprime closure $\sqrt{I}$ is a proper ideal as well.
Thus, every proper ideal in a \ca{} is contained in a prime ideal;
see \autoref{prp:SemiprimeClosure}.

Note that prime ideals are proper by definition.
For further details on prime and semiprime ideals, we refer to \cite[Section~10]{Lam01FirstCourse2ed}.

\medskip

A ring $R$ is said to be a \emph{radical extension} over a subring $S \subseteq R$ if for every $x \in R$ there exists $n \geq 1$ such that $x^n \in S$.
Given an ideal $I$ of a ring $R$, the \emph{semiprime closure} $\sqrt{I}$ is the intersection of all prime ideals containing $I$, with the convention $\sqrt{I}=R$ if there is no prime ideal containing $I$.

\begin{lma}
\label{prp:DominationIdeal}
Let $A$ be a \ca\ and let $I \subseteq A$ be a (not necessarily closed) ideal.
Let $a \in A_+$ and $b \in I_+$ satisfy $a \leq b^{1+\varepsilon}$ for some $\varepsilon > 0$.
Then $a \in I$.
\end{lma}
\begin{proof}
Applying the polar decomposition in \ca{s} (see, for example, \cite[Proposition~II.3.2.1]{Bla06OpAlgs}), there exists $u \in A$ such that $a = ubu^*$.
Since $I$ is an ideal, this implies that $a \in I$.
\end{proof}

\begin{prp}
\label{prp:NotRadicalExtension}
Let $A$ be a \ca, and let $I \subseteq A$ be a (not necessarily closed) ideal such that $A$ is a radical extension over $I$.
Then $I = A$.
\end{prp}
\begin{proof}
It suffices to show that $I$ contains every contractive, positive element in $A$.
Fix $a \in A_+$ with $\| a \| \leq 1$,
and let $f \colon [0,1] \to [0,1]$ be defined as
\[
f(t) := \sum_{n = 1}^\infty \frac{1}{2^n} t^{\frac{1}{n}}
\]
for $t \in [0,1]$.
Note that $f$ is a continuous function with $f(0) = 0$.

Applying continuous functional calculus, we consider the positive contraction $f(a) \in A$.
By assumption, there exists $n \geq 1$ such that $f(a)^n \in I$.
Then the element $2^{2n^2}f(a)^{n}$ belongs to $I$ as well.

We have $\tfrac{1}{2^{2n}}t^{\frac{1}{2n}} \leq f(t)$ and therefore $t \leq (2^{2n^2}f(t)^n)^2$ for every $t \in [0,1]$.
This implies that
\[
a \leq \big( 2^{2n^2}f(a)^{n} \big)^2 \in I.
\]
Applying \autoref{prp:DominationIdeal}, we deduce that $a \in I$.
\end{proof}

\begin{rmk}
\label{rmk:NonSelfadjoint}
The conclusion of \autoref{prp:NotRadicalExtension} is false for non-selfadjoint operator algebras: 
indeed, the nilpotent algebra
$\big\{ \begin{psmallmatrix}
0 & \lambda \\ 
0 & 0 
\end{psmallmatrix} \colon \lambda\in\mathbb{C} \big\}$ 
is a (nonzero) radical extension of $\{0\}$.
\end{rmk}

\begin{thm}
\label{prp:SemiprimeClosure}
Let $I \subseteq A$ be a proper (not necessarily closed) ideal in a \ca{}.
Then $\sqrt{I}$ is a proper ideal of $A$ as well.

In particular, every proper ideal of a \ca{} is contained in a prime ideal.
\end{thm}
\begin{proof}
By \cite[Theorem~10.7]{Lam01FirstCourse2ed}, $\sqrt{I}$ consists of those elements $x \in A$ such that every $m$-system containing~$x$ has nontrivial intersection with $I$, and consequently $\sqrt{I}$ is a radical extension over $I$.
If $\sqrt{I}=A$, then $A$ would be a radical extension over the proper ideal $I$, which is impossible by \autoref{prp:NotRadicalExtension}.
Thus, we have $\sqrt{I} \neq A$.

Since $\sqrt{I}$ is the intersection of prime ideals, it follow that there exists a prime ideal $J \subseteq A$ such that $\sqrt{I} \subseteq J$.
\end{proof}

\begin{cor}
\label{prp:GeneratingSet}
A subset $X$ generates a \ca{} $A$ as a not necessarily closed ideal if and only if $X \nsubseteq I$ for every prime ideal $I \subseteq A$.
\end{cor}

\begin{cor}
\label{prp:MaximalPrime}
In a \ca{}, every prime ideal that is maximal among all prime ideals is also maximal among all ideals.
\end{cor}

We end this section with some remarks and questions related to \autoref{prp:SemiprimeClosure}.

\begin{rmk}
If $I \subseteq A$ is a proper, dense ideal in a \ca{}, then a continued application of \autoref{prp:SemiprimeClosure} provides an ascending chain of prime ideals that continues indefinitely -- unless one hits a maximal, dense ideal.
For some classes of \ca{s}, including the commutative ones, it is known that they contain no maximal, dense ideals, but in general the answer to the following question of Ozawa \cite{Oza:MO17MaxIdealsClosed} remains open:
\end{rmk}

\begin{qst}[Ozawa]
Is every maximal ideal in a \ca{} closed?
\end{qst}

\section{Applications to Lie theory of \texorpdfstring{$C^*$}{C*}-algebras}
\label{sec:Lie}

In this section, we use that every proper ideal in a \ca{} is contained in a prime ideal to transfer some results from the Lie theory of prime rings to the theory of nonclosed Lie ideals in \ca{s}.

\begin{pgr}
\label{pgr:LieThy}
Given subsets $X$ and $Y$ of a a \ca{}, we follow the standard convention and use $[X,Y]$ to denote the liner span of the set of commutators $[x,y] := xy-yx$ for $x \in X$ and $y \in Y$.
Similarly, we use $XY$ to the denote the linear span of the set of products $xy$ for $x \in X$ and $y \in Y$.
Note that $[X,Y]=[Y,X]$.

We will frequently use the Jocobi identity $[[a,b],c]+[[b,c],a]+[[c,a],b]=0$ for $a,b,c \in A$.
It implies that for subsets $X,Y,Z \subseteq A$, we have
\[
\big[ [X,Y],Z \big] 
\subseteq \big[ [Y,Z],X \big] + \big[ [Z,X],Y \big].
\]

A subspace $L \subseteq A$ is said to be a \emph{Lie ideal} if $[A,L] \subseteq L$.
If $K$ and $L$ are Lie ideals in $A$, then so is $[K,L]$.
In particular, $[A,A]$ and $[[A,A],[A,A]]$ are Lie ideals.
\end{pgr}

Given a \ca{}, we let $\widetilde{A}$ denote its minimal unitization.
If $M \subseteq A$ is a subspace, then the not necessarily closed ideal generated by $M$ is $\widetilde{A}M\widetilde{A} = M + AM + MA + AMA$, and in general this is strictly larger than $AMA$ since $M$ may not be contained in $AMA$.
Of course, if $A$ is unital then $\widetilde{A} = A$ and thus $\widetilde{A}M\widetilde{A}=AMA$.
The next result shows that this is also the case if $A$ is nonunital and $M$ is the commutator subspace.

\begin{prp}
\label{prp:CommutatorFactorization}
Let $A$ be a \ca{}.
Then
\[
[A,A] 
\ \subseteq \
A[A,A]A
\ = \ \widetilde{A}[A,A]\widetilde{A}.
\]
\end{prp}
\begin{proof}
It suffices to show the inclusion $[A,A] \subseteq A[A,A]A$.
Let $a,b \in A$.
We need to verify that $[a,b] \in A[A,A]A$.
Using bilinearity of the Lie product, and that every element in $A$ is a linear combination of positive elements, we may assume that $a$ and $b$ are positive.
Using the multiple Cohen factorization theorem (\cite[Theorem~17.1]{DorWic79FactorizationBanachMod}) for $a^{1/2}$ and $b^{1/2}$, we obtain elements $x,y,z \in A$ such that
\[
a^{1/2} = xy, \andSep
b^{1/2} = xz.
\]
Then $a=xyy^*x^*$ and $b=xzz^*x^*$, and thus
\begin{align*}
[a,b]
&= x\big( yy^*x^*xzz^* - zz^*x^*xyy^* \big)x^* \\
&= x\big( yy^*[x^*x,zz^*] + yy^*zz^*x^*x - zz^*x^*xyy^* \big)x^* \\
&= x\big( yy^*[x^*x,zz^*] + [yy^*,zz^*x^*x] \big)x^*,
\end{align*}
which verifies $[a,b] \in A[A,A]A$.
\end{proof}

Given a Lie ideal $L$ in a \ca{} $A$, the next result characterizes when $[L,L]$ generates $A$ as a not necessarily closed ideal.
For the case $L = A$, this question is studied more thoroughly in \cite{GarThi23arX:GenByCommutators}.

\begin{thm}
\label{prp:Lie}
Let $L \subseteq A$ be a Lie ideal in a \ca{}.
Then the following are equivalent:
\begin{enumerate}
\item
$A = A\big[ [L,L],[L,L] \big]A$.
\item
$A = A[L,L]A$.
\item
$A = A[A,L]A$.
\item
$A=\widetilde{A}[A,L]\widetilde{A}$, that is, $[A,L]$ generates $A$ as a not necessarily closed ideal.
\end{enumerate}
Further, if this is the case, then $[A,A] = [L,L] \subseteq L$.
\end{thm}
\begin{proof}
Since $L$ is a Lie ideal in $A$, we have $[[L,L],[L,L]] \subseteq [L,L] \subseteq [A,L]$, which shows that~(1) implies~(2), and that~(2) implies~(3).
Clearly, (3) implies~(4).

To show that~(4) implies~(1), assume that $A = \widetilde{A}[A,L]\widetilde{A}$, and consider the ideal $I := A[[L,L],[L,L]]A$.
To reach a contradiction, assume that $I \neq A$.
Then, by \autoref{prp:SemiprimeClosure}, there exists a prime ideal $J \subseteq A$ containing $I$.

Consider the prime ring $B := A/J$, and let $K$ denote the image of $L$ in $B$.
Then~$K$ is a Lie ideal of $B$ such that $[[K,K],[K,K]]=\{0\}$. 
In particular, $M:=[K,K]$ is a Lie ideal in $B$ such that $[M,M]$ is contained in the center~$Z(B)$ of~$B$.
Then $[K,K] = M \subseteq Z(B)$ by \cite[Lemma~7]{LanMon72LieStrPrimeChar2}.
Applying the same result again for $K$, we get $K \subseteq Z(B)$, and thus $[B,K]=\{0\}$.

On the other hand, we have $A = \widetilde{A}[A,L]\widetilde{A} = A[A,L]A + A[A,L] + [A,L]A + [A,L]$.
Applying the quotient map $A \to B$ everywhere, we get
\[
B = B[B,K]B + B[B,K] + [B,K]B + [B,K] = \{0\}.
\]
which is the desired contradiction.

This shows that~(1)-(4) are equivalent.
Next, we establish the following claim:

Claim~1: \emph{Let $K \subseteq A$ be a Lie ideal such that $A=A[K,K]A$. 
Then $[A,A] \subseteq K$.}
Indeed, it is known that $[A,A[K,K]A] \subseteq K$ holds in general (see, for example, \cite[Lemma~2.1]{Lee22AddSubgpGenNCPoly}), and thus
\[
[A,A] = \big[ A, A[K,K]A \big] \subseteq K.
\]
This proves the claim.

Claim~2: \emph{If $A=A[A,A]A$, then $[A,A] = [[A,A],[A,A]]$.}
Indeed, if $A=A[A,A]A$, then $A=A[[A,A],[A,A]]A$ by the equivalence of~(1) and~(2) applied for the Lie ideal $L=A$.
Applying the same equivalence again, this time for $[A,A]$, we deduce that $A=A[[[A,A],[A,A]],[[A,A],[A,A]]]A$.
Thus, for the Lie ideal $K=[[A,A],[A,A]]$, we have $A=A[K,K]A$, and thus $[A,A] \subseteq K = [[A,A],[A,A]]$ by Claim~1.
The revers inclusion is clear, which proves the claim.

Now, assuming that~(1)-(4) hold for $L$, let us verify that $[A,A] = [L,L] \subseteq L$.
The inclusions $[L,L] \subseteq [A,A]$ and $[L,L] \subseteq L$ are clear.
By Claim~1, we have $[A,A] \subseteq L$.
Further, since $A = A[L,L]A$, we also have $A = A[A,A]A$.
Using Claim~2 at the first step, it follows that
\[
[A,A] 
= \big[ [A,A],[A,A] \big]
\subseteq [L,L],
\]
as desired.
\end{proof}

We highlight the following result, which was established as Claim~2 in the proof of \autoref{prp:Lie}.

\begin{cor}
\label{prp:HigherCommutator}
Let $A$ be a \ca{} that is generated by its commutators as a not necessarily closed ideal.
Then $[[A,A],[A,A]]=[A,A]$.
\end{cor}

It remains an open question if the conclusion of \autoref{prp:HigherCommutator} holds in general:

\begin{qst}
\label{qst:HigherCommutator}
Does $[[A,A],[A,A]]=[A,A]$ hold for every \ca?
\end{qst}

\begin{thm}
\label{prp:LieTwo}
Let $A$ be a \ca{}, and let $K,L \subseteq A$ be Lie ideals.
Then the following are equivalent:
\begin{enumerate}
\item
$A = A[K,L]A$.
\item
$A=\widetilde{A}[K,L]\widetilde{A}$, that is, $[K,L]$ generates $A$ as a not necessarily closed ideal.
\item
$A = A[K,K]A = A[L,L]A$.
\item
$A = \widetilde{A}[K,K]\widetilde{A} = \widetilde{A}[L,L]\widetilde{A}$, that is, $[K,K]$ and $[L,L]$ generate $A$ as a not necessarily closed ideal.
\end{enumerate}
Further, if this is the case, then $[A,A] = [K,L]$.
\end{thm}
\begin{proof}
It is clear that~(1) implies~(2).
By \autoref{prp:Lie}, (3) and~(4) are equivalent.
To show that~(2) implies~(3), assume that $A=\widetilde{A}[K,L]\widetilde{A}$.
Since $[K,L] \subseteq [A,L]$, we deduce that $A=\widetilde{A}[A,L]\widetilde{A}$, and thus $A=A[L,L]A$ by \autoref{prp:Lie}.
Similarly, we get $A=A[K,K]A$.

To show that~(3) implies~(1), assume that $A = A[K,K]A = A[L,L]A$.
Since $[K,K] \subseteq [A,A]$, we obtain that $A = A[A,A]A$ and thus $[A,A] = [[A,A],[A,A]]$ by \autoref{prp:HigherCommutator}.
Further, we have $[A,A] \subseteq K,L$ by the last claim in \autoref{prp:Lie}.
It follows that
\[
[A,A] 
= \big[ [A,A], [A,A] \big] 
\subseteq [K,L] 
\subseteq [A,A].
\]
and thus these inclusions are equalities.
\end{proof}

\section{The subspace generated by square-zero elements in a \texorpdfstring{$C^*$}{C*}-algebra}
\label{sec:N2}

Given a \ca{} $A$, let $N$ denote the subspace generated by the set $N_2(A) := \{ x \in A : x^2 = 0 \}$ of square-zero elements in $A$.
It is know that $N \subseteq [A,A]$ (see, for example, \cite[Lemma~2.1]{Rob16LieIdeals}), and Robert asks in \cite[Question~2.5]{Rob16LieIdeals} if the converse inclusion also holds:

\begin{qst}[{Robert}]
\label{qst:Robert}
Do we have $[A,A] = N$?
Is $N$ a Lie ideal?
\end{qst}

We show that $[N,N] = N$ always holds;
see \autoref{prp:N}.
We further show that $ANA = \widetilde{A}N\widetilde{A} = [A,N]^2$, and while we are not able to show that $N$ is a Lie ideal, we prove that $[A,[A,N]]$ is always a Lie ideal;
see \autoref{prp:IdealGenByN2}.

\begin{thm}
\label{prp:N}
Let $A$ be a \ca{}, and let $N$ denote the subspace generated by the square-zero elements in~$A$.
Then $N=[N,N]$ and $N \subseteq N^2$.
\end{thm}
\begin{proof}
We first show that $[N,N] \subseteq N$.
Given $x,y \in N_2(A)$, we have
\[
[x,y]
= xy-yx
= (1+x)y(1-x) - y + xyx
\]
and each of the last three summands belongs to $N_2(A)$, and so $[x,y] \in N$.

\smallskip 

Next, we verify that $N \subseteq [N,N]$.
Let $x \in N_2(A)$.
Let $p,q \in A^{**}$ be the right and left support projections of $x$, that is, $p$ is the smallest projection in $A^{**}$ satisfying $x=xp$, and similarly for $q$ with $x=qx$.
Consider the polar decomposition $x=v|x|$ in $A^{**}$, with $|x| = (x^*x)^{1/2}$ and $v$ a partial isometry satisfying $q=vv^*$ and $p=v^*v$.

Then $v|x|^{1/2}$, $|x|^{1/2}$ and $|x^*|^{1/2}$ belong to $A$, and we have $|x|^{1/2} = p|x|^{1/2}p$ and $|x^*|^{1/2} = q|x^*|^{1/2}q$.
Using that $x^2=0$ we deduce that $pq=0$ and therefore
\[
|x|^{1/2}|x^*|^{1/2} 
= |x|^{1/2}pq|x^*|^{1/2} 
= 0, \andSep
|x|^{1/2}v
= |x|^{1/2}pqv
= 0.
\]
Further, we have $|x^*|^{1/2}v=v|x|^{1/2}$.
It follows that
\[
x = \left[ \tfrac{1}{2} v|x|^{1/2}, |x|^{1/2} -  |x^*|^{1/2} \right].
\]
Further, we have
\[
|x|^{1/2} -  |x^*|^{1/2}
= \big(|x|^{1/4}v^*\big)\big(v|x|^{1/4}\big) -  \big(v|x|^{1/4}\big)\big(|x|^{1/4}v^*\big).
\]

Since the elements $\tfrac{1}{2} v|x|^{1/2}$, $|x|^{1/4}v^*$ and $v|x|^{1/4}$ have square-zero, it follows that
\[
x \in \big[ N,[N,N] \big] \subseteq [N,N].
\]

\smallskip 

Finally, using that every commutator belongs to $N^2$, we deduce that
\[
N \subseteq [N,N] 
\subseteq N^2,
\]
as desired.
\end{proof}

\begin{rmk}
Note that if $[A,A]=N$, then 
\[\big[[A,A],[A,A]\big] = [N,N] = N = [A,A],\] 
and hence a positive answer to \autoref{qst:Robert} entails a positive answer to \autoref{qst:HigherCommutator}.
\end{rmk}

\begin{thm}
\label{prp:IdealGenByN2}
Let $A$ be a \ca{}, and let $N$ denote the subspace generated by the square-zero elements in~$A$.
Then
\[
ANA = \widetilde{A}N\widetilde{A} = [A,N]^2, \andSep
\big[ A,[A,N] \big]
= \big[ ANA,[A,N] \big]
= [ANA,ANA].
\]
In particular, $[A,[A,N]]$ is a Lie ideal in $A$.
\end{thm}
\begin{proof}
We will repeatedly use that $N=[N,N] \subseteq N^2$ by \autoref{prp:N}.
Since the unit of $\widetilde{A}$ commutes with every element in $A$, we have $[\widetilde{A},X]=[A,X]$ for every subset $X \subseteq A$.

We have $ANA \subseteq \widetilde{A}N\widetilde{A}$.
Conversely, since $\widetilde{A}N \subseteq A$ and $N\widetilde{A} \subseteq A$, we also have
\[
\widetilde{A}N\widetilde{A}
\subseteq \widetilde{A}N^3\widetilde{A}
\subseteq ANA.
\]
Thus, $ANA = \widetilde{A}N\widetilde{A}$.
The inclusion $[A,N]^2 \subseteq \widetilde{A}N\widetilde{A}$ is clear.
We now show the converse inclusion.

Given $a,b \in \widetilde{A}$ and $x,y \in N_2(A)$, we have $a[x,y]b = a[[x,y],b] + ab[x,y]$ and thus
\[
\widetilde{A}N\widetilde{A}
= \widetilde{A}[N,N]\widetilde{A}
\subseteq \widetilde{A}\big[ [N,N],\widetilde{A} \big] + \widetilde{A}[N,N].
\]
Using that $[[N,N],\widetilde{A}] = [N,\widetilde{A}] = [N,A] = [A,N]$, and $[N,N] \subseteq [A,N]$, we get 
\begin{equation}
\label{prp:IdealGenByN2:eq1}
\widetilde{A}N\widetilde{A} 
\subseteq \widetilde{A}[A,N].
\end{equation}

Further, using the Jacobi identity at the second step, we have
\begin{equation}
\label{prp:IdealGenByN2:eq2}
[A,N]
= \big[ A,[N,N] \big]
\subseteq \big[ N,[N,A] \big] + \big[ N,[A,N] \big]
= \big[ [A,N],N \big].
\end{equation}

Next, given $a \in \widetilde{A}$ and $x,y \in N_2(A)$, we have $a[x,y] = [ax,y] + [y,a]x$ and thus
\begin{equation}
\label{prp:IdealGenByN2:eq3}
\widetilde{A}[A,N]
= \widetilde{A}\big[ [A,N],N \big]
\subseteq \big[ \widetilde{A}[A,N],N \big] + [N,\widetilde{A}][A,N]
\subseteq [A,N] + [A,N]^2.
\end{equation}

We get
\[
[A,N]
\stackrel{\eqref{prp:IdealGenByN2:eq2}}{\subseteq} \big[ [A,N],N \big]
= \big[ [A,N],[N,N] \big]
\subseteq \big[ [A,N],[A,N] \big]
\subseteq [A,N]^2,
\]
and it follows that
\[
\widetilde{A}N\widetilde{A}
\stackrel{\eqref{prp:IdealGenByN2:eq1}}{\subseteq} \widetilde{A}[A,N]
\stackrel{\eqref{prp:IdealGenByN2:eq3}}{\subseteq} [A,N] + [A,N]^2
\subseteq [A,N]^2.
\]
We have verified that $ANA = \widetilde{A}N\widetilde{A} = [A,N]^2$. 

Given $a,x,y \in A$, we have $[a,xy] = [ax,y] + [ya,x]$, and thus
\[
\big[ A,[A,N] \big]
\subseteq \big[ A,[A,N]^2 \big]
\subseteq \big[ A[A,N],[A,N] \big] + \big[ [A,N]A,[A,N] \big].
\]
Since $[A,N]^2=ANA$ is an ideal, we have $A[A,N],[A,N]A \subseteq [A,N]^2 = ANA$, and thus $[A,[A,N]] \subseteq [ANA, [A,N]]$.
The converse of this inclusion is clear.

Since $[A,N] \subseteq ANA$, we have $[ANA, [A,N]] \subseteq [ANA, ANA]$.
Conversely, using again that $[a,xy] = [ax,y] + [ya,x]$ for $a,x,y \in A$, we have 
\[
[ANA,ANA]
= \big[ [A,N]^2,[A,N]^2 \big]
\subseteq \big[ [A,N]^3,[A,N] \big]
\subseteq \big[ ANA,[A,N] \big],
\]
as desired.
\end{proof}

\begin{qst}
\label{qst:NilN2}
Is every nilpotent element in a \ca{} a finite sum of square-zero elements?
\end{qst}

Robert showed in \cite[Lemma~2.1]{Rob16LieIdeals} that every nilpotent element in a \ca{} $A$ belongs to $[A,A]$.
Therefore, a positive answer to \autoref{qst:Robert} entails a positive answer to \autoref{qst:NilN2}.

We also recover the following result, which is implicit in \cite{Rob16LieIdeals}.

\begin{cor}
\label{prp:ClosureHigherCommutators}
Let $A$ be a \ca{}.
Then $\overline{[[A,A],[A,A]]} = \overline{[A,A]}$.
\end{cor}
\begin{proof}
Let $N$ denote the subspace generated by the square-zero elements in~$A$.
We have $\overline{N} = \overline{[A,A]}$;
see \cite[Proposition~2.2]{AlaExtVilBreSpe16CommutatorsSquareZero}, \cite[Corollary~2.3]{Rob16LieIdeals}.
Using this at the first and last step, and using \autoref{prp:N} at the second step, we get
\[
\overline{\big[ [A,A],[A,A] \big]} 
= \overline{[N,N]} 
= \overline{N} 
= \overline{[A,A]},
\]
as desired.
\end{proof}


\providecommand{\etalchar}[1]{$^{#1}$}
\providecommand{\bysame}{\leavevmode\hbox to3em{\hrulefill}\thinspace}
\providecommand{\noopsort}[1]{}
\providecommand{\mr}[1]{\href{http://www.ams.org/mathscinet-getitem?mr=#1}{MR~#1}}
\providecommand{\zbl}[1]{\href{http://www.zentralblatt-math.org/zmath/en/search/?q=an:#1}{Zbl~#1}}
\providecommand{\jfm}[1]{\href{http://www.emis.de/cgi-bin/JFM-item?#1}{JFM~#1}}
\providecommand{\arxiv}[1]{\href{http://www.arxiv.org/abs/#1}{arXiv~#1}}
\providecommand{\doi}[1]{\url{http://dx.doi.org/#1}}
\providecommand{\MR}{\relax\ifhmode\unskip\space\fi MR }
\providecommand{\MRhref}[2]{%
  \href{http://www.ams.org/mathscinet-getitem?mr=#1}{#2}
}
\providecommand{\href}[2]{#2}

\end{document}